\documentclass[a4paper,12pt]{amsart}
\usepackage{latexsym,amscd,amssymb,url}
\usepackage[french,english]{babel}
\usepackage[all,cmtip]{xy}  
\usepackage{graphicx}
\usepackage{xcolor}
\usepackage{hyperref}
\usepackage{enumerate}

\newtheorem{theorem}{Theorem}

\theoremstyle{plain}

\newtheorem{question}[theorem]{Question}

\newtheorem{corollary}[theorem]{Corollary}
\newtheorem{definition}[theorem]{Definition}

\newtheorem{lemma}[theorem]{Lemma}

\newtheorem{proposition}[theorem]{Proposition}
\newtheorem{remark}[theorem]{Remark}

\newcommand{\Q}{\mathbb Q}

\newcommand{\C}{\mathbb C}
\newcommand{\N}{\mathbb N}
\newcommand{\R}{\mathbb R}

\newcommand{\De}{\Delta}

\newcommand{\im}{\operatorname{im}}

\newcommand{\Pic}{\operatorname{Pic}}

\newcommand{\id}{\operatorname{id}}

\newcommand{\dashedlongrightarrow}{\xymatrix@1@=15pt{\ar@{-->}[r]&}}
\renewcommand{\longrightarrow}{\xymatrix@1@=15pt{\ar[r]&}}
\renewcommand{\mapsto}{\xymatrix@1@=15pt{\ar@{|->}[r]&}}
\renewcommand{\twoheadrightarrow}{\xymatrix@1@=15pt{\ar@{->>}[r]&}}
\newcommand{\hooklongrightarrow}{\xymatrix@1@=15pt{\ar@{^(->}[r]&}}
\newcommand{\congpf}{\xymatrix@1@=15pt{\ar[r]^-\sim&}}
\renewcommand{\cong}{\simeq}

\begin{document}  
 
\title[On the number and boundedness of log minimal models of general type]{On the number and boundedness of log minimal models of general type}  

\author{Diletta Martinelli}
\address{University of Edinburgh, 
School of Mathematics,
James Clerk Maxwell Building,
The King's Buildings,
Peter Guthrie Tait Road,
Edinburgh
EH9 3FD,
UK.} 
\email{Diletta.Martinelli@ed.ac.uk}

\author{Stefan Schreieder}
\address{Mathematisches Institut, Universit\"at Bonn, Endenicher Allee 60, 53115 Bonn, Germany.} 
\curraddr{Mathematisches Institut, LMU M\"unchen, Theresienstr.\ 39, 80333 M\"unchen, Germany.}
\email{schreieder@math.lmu.de}

\author{Luca Tasin} 
\address{Universit\`a Roma Tre, Dipartimento di Matematica e Fisica, Largo San Leonardo Murialdo I-00146 Roma, Italy} 
\email{ltasin@mat.uniroma3.it}

\date{October 1st, 2018 
}



\maketitle

\begin{abstract}      
We show that the number of marked minimal models of an $n$-dimensional smooth complex projective variety of general type can be bounded in terms of its volume, and, if $n=3$, also in terms of its Betti numbers. 
For an $n$-dimensional projective klt pair $(X,\Delta)$ with $K_X+\Delta$ big, we show more generally that the number of its weak log canonical models can be bounded in terms of the coefficients of $\Delta$ and the volume of $K_X+\Delta$. 
We further show that all $n$-dimensional projective klt pairs $(X,\Delta)$, such that $K_X+\Delta$ is big and nef of fixed volume and such that the coefficients of $\Delta$ are contained in a given DCC set, form a bounded family. 
It follows that in any dimension, minimal models of general type and bounded volume form a bounded family. 
\end{abstract}
%

\section{Introduction}

\subsection{Number of minimal models}
It is well known that starting from dimension three, a minimal model of a smooth complex projective variety $X$ is in general not unique.  
Nevertheless, if $X$ is of general type, even the number of marked minimal models of $X$ is finite \cite{BCHM,KM87}; that is, up to isomorphism, there are only finitely many pairs $(Y,\phi)$, where $\phi \colon X\dashrightarrow Y$ is a birational map and $Y$ is a minimal model, cf.\ Section \ref{subsec:counting} below.  
Such a finiteness statement fails if $X$ is not of general type \cite[Example 6.8]{reid}.  
However, it is conjectured that the number of minimal models $Y$ of $X$ is always finite up to isomorphism; 
this is known for threefolds of positive Kodaira dimension \cite{kawamata}.

In this paper we study the number of marked minimal models in families. 
In particular, 
we show that the corresponding function on any moduli space of complex projective varieties of general type is constructible; that is, it is constant on the strata of some stratification by locally closed subsets. 

\begin{theorem} \label{thm:stratification}  
Let $\pi \colon \mathcal X\longrightarrow B $ be a family of complex projective varieties such that the resolution of each fibre is of general type. 
Then the function $f \colon B\longrightarrow \N$, which associates to $b\in B$ the number of marked minimal models of the fibre $X_b$, is constructible in the Zariski topology of $B$. 
\end{theorem}
 
In contrast to the above theorem, recall that the Picard number is in general not a constructible function on the base of families of varieties of general type.  

Since  smooth complex projective varieties of general type, of given dimension and bounded volume form a  birationally bounded family \cite{HM06,takayama,tsuji},  Theorem  \ref{thm:stratification} implies the following.
 
\begin{corollary}\label{cor:nrofmodels}
Let $n\in \N$ and $c\in \R_{>0}$.  
Then there is a positive constant $N(c)$, such that for any $n$-dimensional smooth complex projective variety $X$ of general type and volume $\operatorname{vol}(X)\leq c$, the number of marked minimal models of $X$ is at most $N(c)$.  
\end{corollary}

One of our original motivations for Theorem \ref{thm:stratification} (resp.\ Corollary \ref{cor:nrofmodels}) stems from \cite{cascini-lazic}, where Cascini and Lazi\'c proved that the number of log minimal models of  a certain class of three-dimensional log smooth pairs $(X,\Delta)$ of general type can be bounded by a constant that depends only on the homeomorphism type of the pair $(X,\Delta)$.
This has its roots in earlier results that show that the topology governs the birational geometry to some extent, see for instance \cite{CZ14} and \cite{kollar-annals}.
Motivated by their result, Cascini and Lazi\'c \cite{cascini-lazic} conjectured that  the number of minimal models of a smooth complex projective threefold of general type is bounded in terms of the underlying topological space.
As an immediate consequence of Corollary \ref{cor:nrofmodels} and \cite[Theorem 1.2]{cascini-tasin}, we solve this conjecture. 
 
\begin{corollary} \label{cor:betti} 
The number of marked minimal models of a smooth complex projective threefold of general type can be bounded  in terms of its Betti numbers. 
\end{corollary}

\subsection{The case of klt pairs} 
A weak log canonical model of a projective klt pair $(X,\Delta)$ is a $(K_X+\Delta)$-non-positive birational contraction 
$
f\colon (X,\Delta)\dashrightarrow (Y,\Gamma=f_\ast \Delta),
$ 
such that $(Y,\Gamma)$ is klt and $K_Y+\Gamma$ is nef.  
If $K_X+\Delta$ is big, then there are only finitely many such models by \cite{BCHM}.
If $\Delta=0$ and $X$ is smooth (or terminal), then any marked minimal model of $X$ is also a weak log canonical model in the above sense.
The converse is not true, because weak log canonical models are not assumed to be $\Q$-factorial; in particular, the number of weak log canonical models is in general larger than the number of marked minimal models.
Theorem \ref{thm:stratification} generalizes to families of klt pairs as follows. 

\begin{theorem} \label{thm:stratification:pairs}
Let $\pi \colon (\mathcal X,\Delta)\longrightarrow B $ be a projective family of klt pairs $(X_b,\Delta_b)$ with $K_{X_b}+\Delta_b$ big. 
Then the function $f \colon B\longrightarrow \N$,  
which associates to $b\in B$ the number of weak log canonical models of $(X_b,\Delta_b)$, 
is constructible in the Zariski topology of $B$.
\end{theorem}

Theorem \ref{thm:stratification:pairs} remains true if we count only those weak log canonical models that are $\Q$-factorial, see Remark \ref{rem:stratification:Q-factorial}. 
Using the log birational boundedness result from \cite{HMX14} (cf.\ Theorem \ref{thm:boundeed:HMX} below), 
Theorem \ref{thm:stratification:pairs} implies the 
following generalization of Corollary \ref{cor:nrofmodels}.

\begin{corollary} \label{cor:nrofmodels:pairs}
The number of weak log canonical models of a projective klt pair $(X,\Delta)$ with $K_X+\Delta$ big,
is bounded in terms of the dimension of $X$, the coefficients of $\Delta$ and the volume of $K_X+\Delta$. 
\end{corollary}

\subsection{Boundedness of log minimal models of general type} \label{subsec:bdd} 
Let $\mathfrak F$ be a collection of projective pairs $(X,\Delta)$.
We recall that the pairs $(X,\Delta)\in \mathfrak F$ form a bounded family (or that $\mathfrak F$ is bounded), if there is a complex projective family of pairs $\pi \colon (\mathcal X,\De) \longrightarrow B$ over a scheme $B$ of finite type, whose fibres belong to $\mathfrak F$ and such that any element of $\mathfrak F$ is isomorphic to some fibre of $\pi$.   
We call $\pi$ a parametrizing family of $\mathfrak F$.

Hacon, M$^c$Kernan and Xu proved the boundedness of the set $\mathfrak F_{slc}$ of all semi log canonical pairs $(X,D)$, where $X$ has given dimension, the coefficients of $D$ belong to a DCC set $I\subset [0,1]$, $K_X+D$ is ample and $(K_X+D)^n$ is fixed, see \cite{HMX-Aut,HMX-ACC,HMX14}.
Here the DCC condition on $I$ means that any non-increasing sequence in $I$ becomes stationary at some point; this holds in particular for any finite set $I\subset [0,1]$.

As a consequence of our study of (log) minimal models in families, we obtain the following partial generalisation of that result.
While we require the pair $(X,D)$ to be klt, we relax the condition on $K_X+D$ to be only big and nef; our result relies on the boundedness theorem from \cite{HMX14}. The two-dimensional case goes back to Alexeev \cite{Alexeev}.

\begin{theorem}\label{thm:boundednesspairs}  
Let $n$ be a natural number, $c$ a positive rational number and $I \subset [0,1) \cap \Q$ be a 
DCC set. 
Consider the set $\mathfrak F$ of all klt pairs $(X,D)$ such that
\begin{enumerate}
\item $X$ is a projective variety of dimension $n$,
\item the coefficients of $D$ belong to $I$,
\item $K_X+D$ is big and nef,  
\item $(K_X+D)^n = c$. 
\end{enumerate} 
Then the pairs $(X,D)\in \mathfrak F$ form a bounded family. 
Moreover, the parametrizing family can be chosen as a disjoint union $ (\mathcal X^{Q},D^{Q}) \sqcup (\mathcal X^{\neq Q},D^{\neq Q})$, where $(\mathcal X^{Q},D^{Q})$ is $\Q$-factorial and parametrizes exactly the $\Q$-factorial members of $\mathfrak F$. 
\end{theorem}

By Theorem \ref{thm:boundednesspairs}, the subset $\mathfrak F^{Q}\subset \mathfrak F$ of $\Q$-factorial pairs is also bounded.
We do not deduce this directly from the boundedness of $\mathfrak F$, but use some refined arguments to reduce to the fact that $\Q$-factoriality is an open condition for families of terminal varieties \cite{KM92}. 

Theorem \ref{thm:boundednesspairs} has several interesting corollaries, which we collect next. 

As it is for instance explained by Hacon and Kov\'acs in \cite[p.\ 9]{HK10}, the boundedness result for canonical models of surfaces implies that  in dimension two, minimal models  of general type and bounded volume form also a bounded family.
That argument fails in higher dimensions, because it uses in an essential way that minimal models of surfaces are unique.   
The following corollary settles the boundedness question for minimal models of general type in arbitrary dimensions. 

\begin{corollary}   \label{cor:minmodbdd}
Minimal models of general type, of given dimension and bounded volume form a bounded family. 
\end{corollary}

Using a suitable Whitney stratification \cite{verdier}, Theorem \ref{thm:boundednesspairs} implies the following finiteness statement for the topological types of minimal models of general type.

\begin{corollary} \label{cor:finitetopspace}    
In the notation of Theorem \ref{thm:boundednesspairs},
 the set of pairs $(X,D)\in \mathfrak F$ is finite up to homeomorphisms of the underlying complex analytic pair $(X^{an},D^{an})$.

Similarly, minimal models of general type, of given dimension and bounded volume are finite up to homeomorphisms of the underlying complex analytic spaces. 
\end{corollary}

By Corollary \ref{cor:finitetopspace}, all topological invariants (e.g.\ the Betti numbers) of a minimal model of general type can be bounded in terms of its volume and dimension.

Corollary \ref{cor:minmodbdd} shows furthermore that the deformation type of a complex projective manifold $X$ with $K_X$ nef and $c_1^n(X)\neq 0$ is determined up to finite ambiguity by the Chern number $c_1^n(X)$:

\begin{corollary} \label{cor:finitedeftype}
Let $n\in \N$ and $c\in \R_{>0}$.  
The set of deformation types of $n$-dimensional complex projective manifolds $X$ with $K_X$ nef and $0<K_X^n\leq c$ is finite.

In particular, any Chern or Hodge number of an $n$-dimensional complex projective manifold $X$ with $K_X$ nef and $c_1^n(X)\neq 0$ can be bounded in terms of $c_1^n(X)$.  
\end{corollary}
  
Under the stronger assumption that $K_X$ is ample, Catanese and Schneider 
\cite{catanese-schneider} have previously shown that any Chern number of the cotangent bundle of $X$ can be bounded from above by some constant multiple of $K_X^n$.

By Koll\'ar's effective base point free theorem \cite{kollar-effectivebpf}, Theorem \ref{thm:boundednesspairs} implies the following boundedness result for Calabi--Yau varieties, which answers a question posed to us by V.\ Tosatti.

\begin{corollary} \label{cor:CY}
Let $n$ be a natural number and $c$ a positive constant.
Then the family of all $n$-dimensional complex projective varieties $X$ with  klt singularities,  
$K_X\equiv 0$ and with a big and nef line bundle $L\in \Pic(X)$ such that $L^n\leq c$, is bounded.
\end{corollary}

\section{Preliminaries}
\subsection{Conventions and notation} 
We work over an algebraically closed field $\mathbf{k}$ of characteristic zero (usually $\mathbf{k}=\C$).
Schemes are assumed to be separated and of finite type over $\mathbf{k}$; varieties are integral schemes. 
A \emph{curve} is a projective scheme of pure dimension one.  

A \emph{family} is a proper flat morphism of schemes $\pi \colon \mathcal X\longrightarrow B$; 
in this paper, the base $B$ will always assumed to be reduced. 
A \emph{projective family} is a family as above such that $\pi$ is a projective morphism.
If not specified otherwise, a fibre of a family is a fibre over a closed point; if $b\in B$ is such a point, then the fibre of $\pi \colon \mathcal X\longrightarrow B$ above $b$ is denoted by $X_b$. 
If $\mathfrak C$ denotes a class of schemes (e.g.\ projective varieties), then a \emph{family of $\mathfrak C$-schemes} is a family $\pi \colon \mathcal X\longrightarrow B$ such that each fibre is in the class $\mathfrak C$.

A \emph{log pair} $(X,\De)$ is a pair where $X$ is a normal variety and $\De=\sum a_i D_i$ is an effective $\Q$-divisor on $X$ such that $0 \le a_i \le 1$ and $K_X+\De$ is $\Q$-Cartier. 
A \emph{family of pairs} $\pi \colon (\mathcal X,\De) \longrightarrow B$ is a family of varieties $\pi \colon \mathcal X\longrightarrow B$ such that $\De$ is an effective $\Q$-divisor on $\mathcal X$ and each fibre $(X_b,\De_b)$ is a log pair. 

If $\pi \colon \mathcal X\longrightarrow B$ is a projective family, and $C\subset X_b$ is a curve in some fibre of $\pi$, then $[C]_{\mathcal X/B}$ denotes the class of $C$ in the relative space of curves $N_1(\mathcal X/B)$, whereas $[C]$ denotes its class in $N_1(X_b)$.

\subsection{Minimal models and weak log canonical models} \label{subsec:counting} 
For the terminology concerning the singularities of a log pair $(X,\De)$ we refer to \cite[Definition 2.34]{kollar-mori}.

A \emph{minimal model} is a projective variety $Y$ with terminal and $\Q$-factorial singularities such that $K_Y$ is nef. 
A \emph{minimal model} of a variety $X$ is a minimal model $Y$ which is birational to $X$. 
A \emph{marked minimal model} of a variety $X$ is a pair $(Y,\phi)$, where $\phi \colon X\dashrightarrow Y$ is a birational map and $Y$ is a minimal model.
(Note that here we do not assume that $\phi$ is a birational contraction.)
Two marked minimal models $(Y,\phi)$ and $(Y',\phi')$ of $X$ are isomorphic if there is an isomorphism $\psi \colon Y\stackrel{\sim}\longrightarrow Y'$ such that $\phi'=\psi\circ \phi$.

By the \emph{number of marked minimal models} of $X$, we mean the number of marked minimal models of $X$ up to isomorphism.
This number is greater or equal than the number of minimal models of $X$, which is obtained by identifying two marked minimal models $(Y,\phi)$ and $(Y',\phi')$ of $X$ if $Y$ and $Y'$ are isomorphic, without asking any compatibility with $\phi$ and $\phi'$. 

A \emph{weak log canonical model} of a projective klt pair $(X,\Delta)$ is a $(K_X+\Delta)$-non-positive birational contraction $f\colon (X,\Delta)\dashrightarrow (Y,\Gamma=f_\ast \Delta)$ such that $(Y,\Gamma)$ is klt and $K_Y+\Gamma$ is nef (see \cite[Definition 3.6.7]{BCHM}).
The \emph{number of weak log canonical models} of $(X,\Delta)$ is the number of weak log canonical models of $(X,\Delta)$ up to isomorphism. 

\subsection{Volume and log canonical models}\label{subsec:canonical-models} 
A \emph{canonical model} is a projective canonical variety $X$ such that $K_X$ is ample.

The \emph{log canonical model} of a projective lc pair $(X,\Delta)$ is a $(K_X+\Delta)$-non-positive birational contraction $f\colon (X,\Delta)\dashrightarrow (Y,\Gamma=f_\ast \Delta)$ such that $(Y,\Gamma)$ is lc and $K_Y+\Gamma$ is ample, cf.\ \cite[Definition 3.6.7]{BCHM}.

Let $(X,\De)$ be an $n$-dimensional projective log canonical pair.
The \emph{volume} of $(X,\De)$ is given by 
$$
\operatorname{vol} (X,\De)= \limsup_{m\to \infty} \frac{n!\cdot h^{0}(X,m(K_X+\De))}{m^n}.
$$ 
If $K_X+\De$ is nef, then $\operatorname{vol}(X,\De)=(K_X+\De)^n$.
If $X$ is canonical, then its volume $\operatorname{vol}(X):=\operatorname{vol}(X,0)$ is a birational invariant of $X$.

By definition, $(X,\De)$ is of \emph{general type} if and only if $\operatorname{vol} (X,\Delta)>0$.

\subsection{Boundedness of log canonical models}
Theorem \ref{thm:boundednesspairs} will be based on the following boundedness result of Hacon, M$^c$Kernan and Xu. 
  
\begin{theorem}[{\cite[Theorem 1.1]{HMX14}}]\label{thm:boundeed:HMX}
Let $n$ be a natural number, $c$ a positive constant and let $I \subset [0,1) \cap \Q$ 
be a DCC set.  
Consider the set $\mathfrak F_{klt}$ of all klt pairs $(X,D)$ such that
\begin{enumerate}
\item $X$ is a projective variety of dimension $n$, 
\item the coefficients of $D$ belong to $I$,
\item $K_X+D$ is ample, 
\item $(K_X+D)^n= c$. 
\end{enumerate} 
Then the pairs $(X,\Delta)\in \mathfrak F_{klt}$ form a bounded family. 
\end{theorem}

\subsection{Flops}\label{subsec:flops} 
Let $(X,\De)$ be a terminal pair.  
A \emph{flop} 
(cf.\ \cite[Definition 6.10]{kollar-mori}) is a birational map $g  \colon  (X,\De) \dashrightarrow (X^+,\De^+=g_*\De)$ to a normal variety $X^+$ which fits into a commutative diagram 
$$
\xymatrix{
		X \ar@{-->}[rr]^g  \ar[dr]_f & & X^+\ar[dl]^{f^+} \\
		& Z 
	}
$$
where $f$ and $f^+$ are small proper birational morphisms to a normal variety $Z$, such that
\begin{enumerate}
\item $f$ is the contraction of an extremal $(K_{X}+\De)$-trivial ray $R\subset \overline{NE}(X)$; 
\item there is a $\Q$-Cartier divisor $H$ on $X$ such that $-H$ is $f$-ample and the pushforward $H^+:=g_\ast H$ is an $f^+$-ample $\Q$-Cartier divisor. 
\end{enumerate} 

The flop   $g \colon (X,\De) \dashrightarrow (X^+,\De^+)$ is uniquely determined by $R$, see \cite[Corollary 6.4]{kollar-mori}.

If $\pi \colon (X,\De) \longrightarrow B$ is a projective morphism and $f \colon X \longrightarrow Z$ is a small contraction corresponding to a relative $(K_X+\De)$-trivial extremal ray $R \subset \overline{NE}(X/B)$, one can define a flop of $R$ in an analogous way.
In this case, the above diagram is defined over $B$, because $f$ is.

A \emph{sequence of flops} is a finite composition of flops; here we allow also the empty sequence, by which we mean an isomorphism.
We will use the following theorem. 

\begin{theorem}(Kawamata  \cite{kawamata-flops}) \label{thm:minmod-connected} 
Let $(Y,\Gamma)$ and $(Y',\Gamma')$ be projective pairs with terminal and $\Q$-factorial singularities such that $K_Y+\Gamma$ and $K_{Y'}+\Gamma'$ are nef. 
Then any birational map $\alpha \colon (Y,\Gamma) \dashrightarrow (Y',\Gamma')$ such that $\alpha_* \Gamma=\Gamma'$ decomposes into a sequence of flops.

\end{theorem}

\begin{remark} \label{rem:flop}
It is well--known that the converse of Theorem \ref{thm:minmod-connected} holds true as well. 
That is, if $(Y,\Gamma)$ is a projective pair with terminal and $\Q$-factorial singularities such that $K_Y+\Gamma$ is nef and $\xi \colon (Y,\Gamma) \dashrightarrow (Y^+,\Gamma^+)$ is a sequence of flops, then $(Y^+,\Gamma^+)$ is a projective pair with terminal and $\Q$-factorial singularities such that $K_{Y^+}+\Gamma^{+}$ is nef.
\end{remark}

In this paper, we will use the following family versions of the above notions. 

\begin{definition} \label{def:familyflop}
Let $B$ be a variety.
A flop of a family $\pi \colon (\mathcal X,\De) \longrightarrow B$ of terminal pairs is a rational map $G\colon (\mathcal X,\De) \dashrightarrow (\mathcal X^+,\De^+)$ over $B$ which 
is defined at the generic point of each fibre and which satisfies the following:
there is a flat family of curves $\mathcal C\subset \mathcal X$ over $B$, such that for any $b\in B$ the curve $C_b$ spans an extremal $(K_{X_b}+\De_b)$-trivial ray of $\overline{NE}(X_b)$, 
and the restriction of $G$ to the fibre above $b$ is a flop of that ray. 

A sequence of flops of families is a finite composition of flops of families; here we also allow the empty sequence, by which we mean an isomorphism 
over $B$.
\end{definition}
 
Clearly, any base change of a flop of families is again a flop of families.
Moreover, it follows from Lemma \ref{lem:KM92} below that a flop of a projective family of terminal $\Q$-factorial pairs is uniquely determined by its restriction to any fibre.

\section{Deformations of $(K_X + \Delta)$-trivial extremal rays}
In this section we prove a deformation result for $(K_{X}+\Delta)$-trivial extremal rays of general fibres of families of log minimal models of general type, see Proposition \ref{prop:raydeformation} below.
For this we will need some auxiliary results, which we collect in Sections \ref{subsec:1-cycle} and \ref{subsec:Hodge}.  

\subsection{$1$-cycles and monodromy} \label{subsec:1-cycle}
Let $\pi\colon\mathcal X\longrightarrow B$ be a family of varieties.
A flat family of $1$-cycles $\alpha$ on $\mathcal X$ over $B$ is a finite linear combination $\alpha=\sum_i a_i\mathcal C_i$, where $a_i\in \R$ and $\mathcal C_i\subset \mathcal X$ is a flat family of curves over $B$.
We denote by $\alpha_b\in N_1(X_b)$ the class of the fibre of $\alpha$ above $b\in B$.

\begin{lemma}[Koll\'ar--Mori] \label{lem:KM92}  
Let $\pi\colon\mathcal X\longrightarrow B$ be a projective family of complex projective varieties over a variety $B$, such that for all $b\in B$, the fibre $X_b$ is $\Q$-factorial and has only terminal singularities, and let $\alpha$ be a flat family of $1$-cycles on $\mathcal X$ over $B$. 
If for some $b_0\in B$, $\alpha_{b_0}=0$ in $N_1(X_{b_0})$, then $\alpha_{b}=0$ in $N_1(X_b)$ for all $b\in B$. 
\end{lemma}

\begin{proof}
Since the fibres of $\pi$ are terminal, they are smooth in codimension two and have only rational singularities, see \cite[Corollary 5.18 and Theorem 5.22]{kollar-mori}.
By \cite[Remark 12.2.1.4.2]{KM92}, \cite[Condition 12.2.1]{KM92} is satisfied and so the lemma follows from \cite[Proposition 12.2.6]{KM92}, where we note that the arguments used work not only for $1$-cycles with rational, but also with real coefficients. 
\end{proof}

\begin{lemma}[Koll\'ar--Mori] \label{lem:monodromy}
Let $B$ be a complex variety and let $\pi\colon\mathcal X\longrightarrow B$ be a projective family of terminal and $\Q$-factorial varieties.
After replacing $B$ by a finite \'etale covering,
the restriction map $N^1(\mathcal X)\longrightarrow N^1(X_t)$ is onto for any very general $t\in B$.
\end{lemma}

\begin{proof}
As in the proof of Lemma \ref{lem:KM92}, our assumptions imply that $\pi$ satisfies \cite[Condition 12.2.1]{KM92}.
The assertion follows therefore from \cite[Proposition 12.2.5]{KM92}.
\end{proof}

\subsubsection{The local system of flat $1$-cycles} \label{subsubsec:monodromy}
Let $B$ be a complex variety and let $\pi\colon\mathcal X\longrightarrow B$ be a projective family of terminal  and $\Q$-factorial varieties. 
Following Koll\'ar and Mori \cite[Definition 12.2.7]{KM92}, there is a local system $\mathcal {GN}_1(\mathcal X/B)$ in the analytic topology of $B$, which maps an analytic open subset $U\subseteq B$ to the set of flat $1$-cycles on $\mathcal X\times_B U$ over $U$ with real coefficients and modulo fibrewise numerical equivalence.
Note that in \cite{KM92} rational coefficients are used, but all results hold also for real coefficients, cf.\ \cite[p.\ 116]{defernex-hacon}. 
For very general $t\in B$, $N_1(X_t)= \mathcal {GN}_1(\mathcal X/B)_t$, and so there is a natural monodromy action on $N_1(X_t)$, see \cite[Proposition 12.2.8]{KM92}.
By \cite[Corollary 12.2.9]{KM92}, this monodromy action is trivial (i.e.\ the local system $\mathcal {GN}_1(\mathcal X/B)$ is trivial) if and only if $N^1(\mathcal X)\longrightarrow N^1(X_t)$ is onto for very general $t\in B$.

By \cite[Theorem 5.22]{kollar-mori}  and \cite[12.1.5.2]{KM92}, numerical and homological equivalence coincide for $1$-cycles on the fibres of $\pi$.
Therefore, $\mathcal {GN}_1(\mathcal X/B)$ can be identified with a subsheaf of $(R^2\pi_{\ast} \R)^{\vee}$. 
If $\pi$ is smooth, then it follows from the equality $N_1(X_t)= \mathcal {GN}_1(\mathcal X/B)_t$ for very general $t\in B$, that this subsheaf coincides with the local subsystem of $(R^2\pi_{\ast} \R)^{\vee}$ which is  generated (over $\R$) by rational sections (i.e.\ sections of $(R^2\pi_{\ast} \Q)^{\vee}$) that are pointwise of Hodge type $(-1,-1)$. 

\subsection{Hodge structures and singularities}
\label{subsec:Hodge}
Recall that by the Lefschetz $(1,1)$-theorem, for any smooth complex projective variety $\widetilde X$, a class in $H^2(\widetilde X,\Q)$ is algebraic if and only if it is of Hodge type $(1,1)$.
This explains the significance of the following result. 

\begin{lemma} \label{lem:MHS}
Let $X$ be a complex projective variety with only rational singularities, and let $\tau\colon \widetilde X\longrightarrow X$ be a resolution of singularities.
Then the cokernel of the natural morphism
$\tau^{\ast} \colon H^2(X,\Q)\longrightarrow H^2(\widetilde X,\Q)$ 
is a pure Hodge structure of type $(1,1)$.  
\end{lemma}

\begin{proof}
Since $\tau^{\ast}$ is a morphism of mixed Hodge structures and $\widetilde X$ is smooth,  $\im(\tau^{\ast})$ is a (pure) sub Hodge structure of $H^2(\widetilde X,\Q)$.  
By classical Hodge theory, it therefore suffices to prove that the composition
$$
H^2(X,\C)\stackrel{\tau^{\ast}}\longrightarrow H^2(\widetilde X,\C) 
\twoheadrightarrow
H^{0,2}(\widetilde X)\cong H^2(\widetilde X,\mathcal O_{\widetilde X})
$$
is onto. 
Since $X$ has only rational singularities, the Leray spectral sequence induces an isomorphism
$
H^2(\widetilde X,\mathcal O_{\widetilde X}) \cong H^2(X,\mathcal O_X)
$, 
and the natural map $H^2(X,\C)\longrightarrow H^2(X,\mathcal O_X) $ is onto, see \cite[Theorem 12.3]{kollar-shafarevich}. 
This concludes the lemma. 
\end{proof}

\begin{lemma}\label{lem:X=terminal}  
Let $B$ be a quasi-projective variety, and let $\pi \colon (\mathcal X,\Delta)\longrightarrow B$ be a family of terminal pairs. 
Then, for any sufficiently small dense open subset $U\subset B$, the base change $\mathcal X_U:=\mathcal X\times_BU$ with boundary $\Delta_U:=\Delta\times_X U$ form a terminal pair $(\mathcal X_U,\Delta_U)$. 
\end{lemma}

\begin{proof}  
Since the fibres of $\pi$ are normal, the normalization of $\mathcal X$ is an isomorphism over a general point of $B$. 
After shrinking $B$, we may thus assume that $\mathcal X$ is normal.
The lemma follows then for instance from \cite[Proposition 3.5]{defernex-hacon}.
\end{proof}

\subsection{Deforming $(K_X+\Delta)$-trivial rays}
The next proposition is the key technical result of this paper. 
It can be seen as a weak version of Koll\'ar--Mori's deformation theorem for threefold flops, see \cite[Theorem 11.10]{KM92}. 
While our result holds in any dimension, the original result of Koll\'ar and Mori 
reflects a peculiarity of threefold geometry, cf.\ Remark \ref{rem:raydeformation} below.

\begin{proposition} \label{prop:raydeformation}
Let $B$ be a variety, and let $\pi\colon (\mathcal X,\Delta)\longrightarrow B$ be a projective family of terminal and $\Q$-factorial pairs $(X_b,\Delta_b)$ with $K_{X_b}+\Delta_b$ big and nef.
Suppose that the restriction map $N^1(\mathcal X)\longrightarrow N^1(X_t)$ is onto for very general $t\in U$. 
Then, for any sufficiently small Zariski open and dense subset $U\subset B$, the base change $\pi_U\colon (\mathcal X_U,\Delta_U)\longrightarrow U$ has the following property.

For any point $0\in U$ and any extremal $(K_{X_0}+\Delta_0)$-trivial ray $R\subset \overline{NE}(X_0)$, there is a flat family of curves $\mathcal C\subset \mathcal X_U$ over $U$, such that 
\begin{enumerate}
\item the fibre of $\mathcal C$ above $0$ spans $R$, i.e.\ $[C_0]\in R$; \label{item:prop:raydef:1}
\item for all $b\in U$, the fibre $C_b$ spans an extremal $(K_{X_b}+\Delta_b)$-trivial ray of $\overline{NE}(X_b)$;
\item for all $b\in U$, the fibre $C_b$ spans an extremal $(K_{\mathcal X}+\Delta)$-trivial ray of $\overline{NE}(\mathcal X_U/U)$; 
\item 
if $C_b$ spans a small ray of $\overline{NE}(X_b)$ for one $b\in U$, then the same holds true for all $b\in U$. 
\label{item:prop:raydef:F}
\end{enumerate} 
\end{proposition}

\begin{proof} 
Since the conclusions of Proposition \ref{prop:raydeformation} are stable under shrinking $U$, it suffices to prove the existence of one open and dense subset $U\subset B$ which satisfies Proposition \ref{prop:raydeformation}.
For ease of notation, we will 
suppress the base changes made in our notation. 

We start with some preliminary considerations.
By Lemma \ref{lem:X=terminal}, we may after shrinking $B$ assume that $(\mathcal X,\Delta)$ has only terminal singularities. 
By assumptions, $K_{\mathcal X}+\Delta$ is $\pi$-big and $\pi$-nef and so we can apply the relative base point free theorem \cite[Theorem 3.24]{kollar-mori} to obtain a morphism $p\colon \mathcal X\longrightarrow \mathcal X^{can}$ to the relative log canonical model $\mathcal X^{can}$ of $(\mathcal X,\Delta)$ over $B$. 
For $\Delta^{can}:=p_\ast \Delta$, we consider the natural morphism $\pi^{can}\colon (\mathcal X^{can},\Delta^{can})\longrightarrow B$.
Since $(\mathcal X,\Delta)$ is terminal, $\lfloor \Delta \rfloor=0$.
Hence, $(\mathcal X^{can},\Delta^{can})$ is klt.
By Bertini's theorem, we may then assume after shrinking $B$, that all fibres $(X^{can}_b,\Delta^{can}_b)$ of $\pi^{can}$ are klt.  
After shrinking $B$ further, we can by Verdier's generalization of Ehresmann's lemma assume that $\mathcal X$ and $\mathcal X^{can}$ are topologically trivial over $B$, and so $R^2\pi_{\ast} \Q$ and $R^2\pi^{can}_\ast \Q$ are local systems on $B$, see \cite[Corollaire 5.1]{verdier}.

Let $\tau\colon\widetilde{\mathcal X}\longrightarrow \mathcal X$ be a projective resolution of singularities.
After shrinking $B$, we may assume that the natural morphism $\widetilde{\pi}\colon \widetilde {\mathcal X}\longrightarrow B$ is smooth and so $R^{2}\widetilde \pi_{\ast} \Q$ is also a local system. 
We then proceed in several steps.
\\

\textbf{Step 1.}
For any global section $\gamma\in \Gamma(B, (R^{2}\widetilde \pi_{\ast} \Q)^{\vee})$, which is pointwise of type $(-1,-1)$ (i.e.\ which corresponds pointwise to a class in $N_1(\widetilde X_b)$), there is a flat family of $1$-cycles $\alpha$ on $\mathcal X$ over $B$ such that $\tau_{\ast}\gamma_b=\alpha_b$ in $N_1(X_b)$ for all $b\in B$.

\begin{proof} 
Recall from Section \ref{subsubsec:monodromy} the local system $\mathcal {GN}_1(\mathcal X/B)$ on $B$, which can be identified with a subsheaf of $(R^{2} \pi_{\ast} \R)^{\vee}$. 
We claim that $(\tau^{\ast})^{\vee}(\gamma)\in \Gamma(B, (R^{2} \pi_{\ast} \Q)^{\vee})$, viewed as a section of $(R^{2} \pi_{\ast} \R)^{\vee}$, is contained in that subsheaf.
Indeed, since $R^{2} \pi_{\ast} \R$ is a local system, it suffices to prove this inclusion at a single point of $B$, and so the claim follows from the equality $\mathcal {GN}_1(\mathcal X/B)_t=N_1(X_t)$ for very general $t\in B$ and the fact that $(\tau^{\ast})^{\vee}(\gamma)_t\in N_1(X_t)$, because $\gamma$ is pointwise of type $(-1,-1)$ and so it can be represented by a $1$-cycle.

As we have thus seen, $(\tau^{\ast})^{\vee}(\gamma)\in \Gamma(B,\mathcal {GN}_1(\mathcal X/B))$, and so we obtain a flat family of $1$-cycles $\alpha$ on $\mathcal X$ over $B$, which has the property we want in Step 1. 
\end{proof}

\textbf{Step 2.}   
For any point $0\in B$ and any $(K_{X_0}+\Delta_0)$-trivial curve $C_0$ in the fibre $X_0$, there is a flat family of $1$-cycles $\alpha$ on $\mathcal X$ over $B$ such that $\alpha_0=[C_0]$ in $N_1(X_0)$.

\begin{proof}  
We will use that $R^2\widetilde \pi_{\ast}\Q$, $R^2\pi_{\ast} \Q$ and $R^2\pi^{can}_\ast \Q$ are local systems on $B$, and that the fibres of $\pi^{can}$ have klt singularities.

Let $\widetilde p\colon\widetilde {\mathcal X} \longrightarrow \mathcal X^{can}$ be the composition of the resolution $\tau \colon \widetilde{\mathcal X}\longrightarrow \mathcal X$ with the morphism $p\colon\mathcal X\longrightarrow \mathcal X^{can}$ to the relative log canonical model of $(\mathcal X,\Delta)$ over $B$. 
The kernel
$$
\mathcal H:=\ker \left( (\widetilde p^{\ast})^{\vee}\colon(R^2\widetilde \pi_\ast \Q)^{\vee}\longrightarrow (R^2\pi^{can}_\ast \Q)^{\vee} \right)
$$
is then a local system on $B$.
In fact, since $\widetilde p^{\ast}$ is fibrewise a morphism of mixed Hodge structures and $(R^2\widetilde \pi_\ast \Q)^{\vee}$ is a variation of Hodge structures of weight $-2$, $\mathcal H\subset (R^2\widetilde \pi_\ast \Q)^{\vee}$ is a sub-variation of Hodge structures. 
Since the fibres of $\pi^{can}$ are klt, they have rational singularities, see \cite[Theorem 5.22]{kollar-mori}.  
It therefore follows from Lemma \ref{lem:MHS} that $\mathcal H^{-2,0}=0$ and so $\mathcal H\otimes \C=\mathcal H^{-1,-1}$ is pure of type $(-1,-1)$. 

Since $R^2\pi_{\ast} \Q$ is a local system, the kernel
$$
\mathcal K:=\ker \left( (\tau^{\ast})^{\vee}\colon(R^2\widetilde \pi_\ast \Q)^{\vee}\longrightarrow (R^2\pi_\ast \Q)^{\vee} \right)
$$
is a sub variation of Hodge structures of $\mathcal H$.
Since $\widetilde \pi$ is projective, $\mathcal H$ is polarizable and so we can choose a direct sum decomposition 
\begin{align} \label{eq:K+H'}
\mathcal H=\mathcal K\oplus \mathcal H' ,
\end{align}
where $\mathcal H'\subseteq \mathcal H$ is a certain sub variation of Hodge structures.
In particular, $\mathcal H'$ is pointwise of type $(-1,-1)$ and so it can be identified with a local subsystem of $\mathcal {GN}_1(\mathcal {\widetilde X}/B)$, see Section \ref{subsubsec:monodromy}. 
By construction, $(\tau^\ast)^{\vee}$ acts injectively on this local system and so $\mathcal H'\cong (\tau^\ast)^{\vee} (\mathcal H')$  is isomorphic to a local subsystem of $\mathcal {GN}_1(\mathcal {X}/B)$. 
This implies that $\mathcal H'$ is a trivial local system, because for $t\in B$ very general, the monodromy action on $N_1(X_t)\cong N^1(X_t)^{\vee}$ is trivial by assumptions, see Section \ref{subsubsec:monodromy}. 

Let now $0\in B$ be a point and let $C_0\subset X_0$ be a $(K_{X_0}+\Delta_0)$-trivial curve.
Up to replacing $C_0$ by some positive multiple, we can choose a curve $\widetilde C_0$ on $\widetilde X_0$ with $\tau_\ast \widetilde C_0=C_0$.
The curve $\widetilde C_0$ yields an element in $H^2(\widetilde X_0,\Q)^\vee$ and we note that actually $[\widetilde C_0]\in \mathcal H_0\subset H^2(\widetilde X_0,\Q)^\vee$, because $p_{\ast}( \tau_{\ast}(\widetilde C_0))=0$.
By (\ref{eq:K+H'}), $[\widetilde C_0]$ gives rise to a unique element $\gamma'_0\in \mathcal H_0'$ with $\tau_{\ast}\gamma'_0=[C_0]$ in $N_1(X_0)$.
Since $\mathcal H'$ is trivial, $\gamma_0'$ can be extended to give a global section $\gamma' \in \Gamma(B,\mathcal H')\subseteq \Gamma(B,(R^2\widetilde \pi_{\ast}\Q)^{\vee})$.
By Step 1, this implies that there is a flat family of $1$-cycles $\alpha$ on $\mathcal X$ over $B$ with $\tau_{\ast}\gamma'_b=\alpha_b$ in $N_1(X_b)$ for all $b\in B$.
In particular, $\alpha_0=[C_0]$, which finishes the proof of Step 2.
\end{proof}

Since $N^1(\mathcal X)\longrightarrow N^1(X_t)$ is onto for very general $t\in B$,  there is a natural inclusion
 $\overline{NE}(X_t)\subseteq \overline{NE}(\mathcal X/B)$, which induces an inclusion on the $(K_{\mathcal X}+\Delta)$-trivial parts.
\\

\textbf{Step 3.}
After shrinking $B$, we may assume that, for any very general $t\in B$, the natural inclusion $\overline{NE}(X_t)\subseteq \overline{NE}(\mathcal X/B)$ yields an equality
$$
\overline{NE}(X_t)_{\{K_{X_t}+\Delta_t=0\}}= \overline{NE}(\mathcal X/B)_{\{K_{\mathcal X}+\Delta=0\}}
$$ 
on the $(K_{\mathcal X}+\Delta)$-trivial parts.
Moreover, for any extremal $(K_{\mathcal X}+\Delta)$-trivial ray $R\subset \overline{NE}(\mathcal X/B)$, there is a flat family of curves $\mathcal C\subset \mathcal X$ over $B$, with $[C_b]_{\mathcal X/B}\in R$ for all $b\in B$. 

\begin{proof} 
We start with some well-known preliminary considerations.
Recall first the morphism $p\colon\mathcal X\longrightarrow \mathcal X^{can}$ to the relative log canonical model of $(\mathcal X,\Delta)$ over $B$. 
We then choose an effective $p$-antiample divisor $H$ on $\mathcal X$ and note that up to adding some sufficiently large multiple of the pullback of an ample divisor on $\mathcal X^{can}$, we may assume that $H.C>0$ for any curve $C\subset \mathcal X$ with $p_\ast C\neq 0$. 
It follows that for any $\epsilon >0$, the $K_{\mathcal X}+\Delta$ trivial curves are exactly the $ K_{\mathcal X}+\Delta+\epsilon H$ negative curves.
Since $(\mathcal X,\Delta+\epsilon H)$ is klt for $0<\epsilon <<1$ and $K_{\mathcal X}+\Delta+\epsilon H$ is $\pi$-big, we deduce from the cone theorem \cite[Theorem 3.25(2)]{kollar-mori} that 
there are only finitely many $(K_{\mathcal X}+\Delta)$-trivial extremal rays $R\subset \overline{NE}(\mathcal X/B)$, and for any such extremal ray,  its contraction  $f_R\colon\mathcal X\longrightarrow \mathcal Z$ exists. 

If the exceptional locus $Exc(f_R)$ does not dominate $B$, then the Zariski open and dense subset  $U=B\setminus{\pi(Exc(f_R))}$ of $B$ has the property that $R$ is not contained in $\overline{NE}(\mathcal X_U/U)\subset \overline{NE}(\mathcal X/B)$.
We then replace $B$ by $U$ and repeat this procedure.
We claim that this process terminates, giving rise to a situation where $Exc(f_R)$ dominates $B$ for all $(K_{\mathcal X}+\Delta)$-trivial extremal rays $R\subset \overline{NE}(\mathcal X/B)$.

We prove the claim in the following.   
Fix $\epsilon >0$ so that $(X,\Delta+\epsilon H)$ is klt, where $H$ is as above. 
By our choice of $H$, a curve on $\mathcal X$ is $(K_{\mathcal X}+\Delta+\epsilon H)$-negative, if and only if it is $(K_{\mathcal X}+\Delta)$-trivial.  
Since  $(K_{\mathcal X}+\Delta+\epsilon H)$ is $\pi$-big, it follows from the relative cone theorem that the $(K_{\mathcal X}+\Delta)$-trivial part of $\overline{NE}(\mathcal X/B)$ is a polyhedral cone with finitely many extremal rays.
%
Moreover, any $(K_{\mathcal X}+\Delta)$-trivial extremal ray $R$ of $\overline{NE}(\mathcal X/B)$ contains a rational curve $C\in R$ with $-2\dim(\mathcal X)\leq (K_{\mathcal X}+\Delta+\epsilon H).C <0$, 
 see \cite[Theorem 3.25]{kollar-mori}.
It follows that the intersection of that cone with the subset of points $D \in N_1(\mathcal X/B)$ given by  the condition $-2\dim(\mathcal X)\leq (K_{\mathcal X}+\Delta+\epsilon H).D <0$ is a bounded domain in $\overline{NE}(\mathcal X/B)$, and so it contains only finitely many integral points. 
This shows that the above process terminates, and so, after shrinking $B$, we may assume that for any $(K_{\mathcal X}+\Delta)$-trivial extremal ray $R\subset \overline{NE}(\mathcal X/B)$, the exceptional locus $Exc(f_R)$ dominates $B$.
This implies in particular that the natural inclusion $\overline{NE}(X_t)_{\{K_{X_t}+\Delta_t=0\}}\subseteq \overline{NE}(\mathcal X/B)_{\{K_{\mathcal X}+\Delta=0\}}$ is an equality.

The relative Hilbert scheme of $\mathcal X$ over $B$ exists as a union of countably many projective schemes over $B$, see for instance \cite[Chapter I]{kollar-rationalcurves}.
For any curve $C$ on a very general fibre $X_t$ of $\pi$, there is therefore a finite covering $B'\longrightarrow B$, and a flat family of curves $\mathcal C'\subset \mathcal X':= \mathcal X\times _BB'$ over $B'$, whose fibre above some $t'\in B'$ which maps to $t\in B$ coincides with $C$.
The natural map $h:\mathcal X'\longrightarrow \mathcal X$ is finite (hence proper), and we may thus define the pushforward $\mathcal C:=h_{\ast}\mathcal C'$.
Since the natural monodromy action on $N_1(X_t)$ is trivial (see Section \ref{subsubsec:monodromy}), $[C_t]$ is a positive multiple of $[C]$.
Moreover, by generic flatness, $\mathcal C$ is flat over some Zariski open and dense subset of $B$.
Since for very general $t\in B$,
$$
\overline{NE}(X_t)_{\{K_{X_t}+\Delta_t=0\}}= \overline{NE}(\mathcal X/B)_{\{K_{\mathcal X}+\Delta=0\}} ,
$$
and since this cone has only finitely many $(K_{\mathcal X}+\Delta)$-trivial extremal rays, the above argument shows that, after possibly shrinking $B$, we may assume the following: 
for any $(K_{\mathcal X}+\Delta)$-trivial extremal ray $R\subset \overline{NE}(\mathcal X/B)$, there is a flat family of curves  $\mathcal C\subset \mathcal X$ over $B$, with  $[C_b]_{\mathcal X/B}\in R$ for all $b\in B$. 
This concludes Step 3. 
\end{proof}

It follows from Step 3 that $\overline {NE}(\mathcal X/B)$ is stable under shrinking $B$ further.
Moreover, as we have mentioned in the proof of Step 3, $\overline {NE}(\mathcal X/B)$ has only finitely many $(K_{\mathcal X}+\Delta)$-trivial extremal rays.
After shrinking $B$, we may therefore assume that for any $(K_{\mathcal X}+\Delta)$-trivial extremal ray $R\subset \overline{NE}(\mathcal X/B)$, the restriction $f_R|_{X_b} \colon X_b\longrightarrow Z_b$ of the contraction $f_R \colon \mathcal X\longrightarrow \mathcal Z$ of $R$, to the fibre above $b$, has connected fibres and that $Z_b$ is normal.

From now on, we do not perform any more base changes. 
\\

\textbf{Step 4.} 
Let $0\in B$ be a point and let $C_0\subset X_0$ be a $(K_{X_0}+\Delta_0)$-trivial curve which spans an extremal ray of $\overline{NE}(X_0)$. 
Then there is a flat family of (possibly reducible) curves $\mathcal D\subset \mathcal X$ over $B$ such that the fibre $D_b$ above $b\in B$ satisfies 
\begin{enumerate}
\item $[D_0]\in [C_0]\cdot \R_{>0}$ in $N_1(X_0)$;  
\item for very general $t\in B$, $[D_t]$ spans an extremal ray of $\overline{NE}(X_t)$; \label{item:step4:2}
\item for all $b\in B$, $[D_b]_{\mathcal X/B}$ spans an extremal ray of $\overline{NE}(\mathcal X/B)$.
\end{enumerate}

\begin{proof}
The curve $C_0\subset X_0$ gives rise to an element $[C_0]_{\mathcal X/B}\in \overline{NE}(\mathcal X/B)$ in the relative cone.
Since $C_0$ is $(K_{\mathcal X}+\Delta)$-trivial and $K_{\mathcal X}+\Delta$ is $\pi$-nef, we can write
\begin{align} \label{eq:NE(X/B)}
[C_0]_{\mathcal X/B}=\sum_{i=1}^r a_i [D_i]_{\mathcal X/B} ,
\end{align}
where $D_i\subset \mathcal X$ are curves contracted by $\pi$ which span pairwise different extremal $(K_{\mathcal X}+\Delta)$-trivial rays of $\overline{NE}(\mathcal X/B)$, and $a_i\geq 0$ are nonnegative rational numbers.

By Step 3, we may assume that there are flat families of curves $\mathcal D_i\subset \mathcal X$ over $B$ such that for any $b\in B$, the class of the fibre $[D_{i,b}]_{\mathcal X/B}\in \overline{NE}(\mathcal X/B)$ coincides with $[D_i]_{\mathcal X/B}$. 
We claim that this implies
\begin{align} \label{eq:NE(X)}
[C_0]=\sum_{i=1}^r a_i [D_{i,0}] 
\end{align}
in $N_1(X_0)$.
In order to prove (\ref{eq:NE(X)}), we need to see that both sides have the same intersection numbers with all line bundles on $X_0$.
By (\ref{eq:NE(X/B)}), this is true for those line bundles that extend to $\mathcal X$, that is, for the subspace $\iota^{\ast}N^1(\mathcal X)\subset N^1(X_0)$, where $\iota \colon X_0\longrightarrow \mathcal X$ denotes the inclusion.
By \cite[Corollary 12.2.9]{KM92}, the subspaces  $\iota^{\ast}N^1(\mathcal X)\subset N^1(X_0)$ and $\mathcal{GN}_1(\mathcal X/B)_0\subset N_1(X_0)$ are dual to each other, cf.\ Section \ref{subsubsec:monodromy}.
We thus obtain a decomposition
$$
N^1(X_0)= \iota^{\ast} N^1(\mathcal X)\oplus (\mathcal{GN}_1(\mathcal X/B)_0)^{\perp},
$$
where $(\mathcal{GN}_1(\mathcal X/B)_0)^{\perp}\subset N^1(X_0)$ denotes the subset of all classes that pair to zero with $\mathcal{GN}_1(\mathcal X/B)_0$, the space of  curves on $X_0$ that extend to flat families of $1$-cycles over $B$.
By Step 2, $[C_0]\in \mathcal{GN}_1(\mathcal X/B)_0$ and so both sides in (\ref{eq:NE(X)}) have zero intersection with $ (\mathcal{GN}_1(\mathcal X/B)_0)^{\perp}$, which concludes the proof of (\ref{eq:NE(X)}).

Since $[C_0]\cdot \R_{>0}$ is an extremal ray of $\overline {NE}(X_0)$, it follows from Lemma \ref{lem:KM92} that in equation (\ref{eq:NE(X)}) at most one coefficient $a_i$ is nonzero.
We may therefore without loss of generality assume $r=1$ and $a_1>0$.
The flat family of curves $\mathcal D:=\mathcal D_1\subset \mathcal X$ over $B$ satisfies therefore $[D_0]\in [C_0]\cdot \R_{>0}$ in $N_1(X_0)$ and for any $b\in B$, 
$
[D_b]_{\mathcal X/B}=[D_1]_{\mathcal X/B}
$ 
spans an extremal ray of $\overline {NE}(\mathcal X/B)$.
For very general $t\in B$, it remains to prove that $D_t$ spans an extremal ray of $\overline{NE}(X_t)$.
This follows immediately from the fact that $D_t$ spans an extremal ray of $\overline{NE}(\mathcal X/B)$ and the assumption that any element of $N^1(X_t)$ lifts to an element of $N^1(\mathcal X)$.
This concludes Step 4.
\end{proof}

\textbf{Step 5.} 
Let $0\in B$ be a point and let $C_0\subset X_0$ be a $(K_{X_0}+\Delta_0)$-trivial curve which spans an extremal ray of $\overline{NE}(X_0)$. 
Let $\mathcal D\subset \mathcal X$ be the corresponding flat family of curves from Step 4.
Then, 
\begin{enumerate}
\item  for all $b\in B$, $[D_b]$ spans an extremal ray of $\overline{NE}(X_b)$; 
\item  if $F \colon \mathcal X\longrightarrow \mathcal Z$ denotes the contraction of the extremal ray $[D_0]_{\mathcal X/B}\cdot \R_{>0}\subset \overline{NE}(\mathcal X/B)$, then, for all $b\in B$, the restriction $F_b:=F|_{X_b}$ of $F$ is the contraction of the extremal ray $[D_b]\cdot \R_{>0}\subset \overline{NE}(X_b)$.
\end{enumerate}

\begin{proof}
For a contradiction, suppose that there is a point $b_0\in B$ such that $[D_{b_0}]\in N_1(X_{b_0})$ does not span an extremal ray.
Then the class $[D_{b_0}]$ can be written as a positive linear combination of at least two effective $(K_{X_{b_0}}+\Delta_{b_0})$-trivial curves on $X_{b_0}$ that span pairwise different extremal rays of $\overline{NE}(X_{b_0})$.
Applying Step 4 to all these extremal curves shows that their classes in $N_1(X_{b_0})$ extend to flat families of curves over $B$.
It therefore follows from Lemma \ref{lem:KM92} that for any $b\in B$, $[D_b]$ does not span an extremal ray of $\overline{NE}(X_{b})$.
This contradicts item (\ref{item:step4:2}) of Step 4.

Let now $F \colon \mathcal X\longrightarrow \mathcal Z$ denote the contraction of the extremal ray $[D_0]_{\mathcal X/B}\cdot \R_{>0}\subset \overline{NE}(\mathcal X/B)$.
As explained above Step 4, for any $b\in B$, $Z_b$ is normal and $F_b \colon X_b\longrightarrow Z_b$ has connected fibres.
Let $b_0\in B$, and let $C\subset X_{b_0}$ be a curve contracted by $F_{b_0}$ and which spans an extremal ray of $\overline{NE}(X_{b_0})$.
Then, by Step 4, there is a flat family of curves $\mathcal C\subset \mathcal X$ over $B$, with $[C_{b_0}]\in [C]\cdot \R_{>0}$ in $N_1(X_{b_0})$.
This implies $[C_{b_0}]_{\mathcal X/B}\in [C]_{\mathcal X/B}\cdot \R_{>0}$ in $N_1(\mathcal X/B)$.
Since $\mathcal C$ is flat over $B$, we conclude that for any $b\in B$, the curve $C_b$ satisfies
$$
[C_b]_{\mathcal X/B}\in [C]_{\mathcal X/B}\cdot \R_{>0}=[D_b]_{\mathcal X/B}\cdot \R_{>0}
$$
in $N_1(\mathcal X/B)$, where the last equality follows from the fact that $C$ is contracted by $F$. 
In particular, for very general $t\in B$, $ [C_t]_{\mathcal X/B} \in [D_t]_{\mathcal X/B}\cdot \R_{>0}$.
Since $N^1(\mathcal X)\longrightarrow N^1(X_t)$ is onto for very general $t\in B$, $[C_t]\in [D_t]\cdot \R_{>0}$ in $N_1(X_t)$.
It then follows from Lemma \ref{lem:KM92} that $[C_b]\in [D_b]\cdot \R_{>0}$ in $N_1(X_b)$ for all $b\in B$.
In particular, $[C]\cdot \R_{>0}= [D_{b_0}]\cdot \R_{> 0}$, because $[C_{b_0}]$ is contained in both rays.  
We have thus proven that any extremal curve contracted by $F_{b_0}$ is numerically proportional to $D_{b_0}$ and so $F_{b_0}$ is indeed the contraction of the extremal ray $[D_{b_0}]\cdot \R_{> 0}$, as we want. 
This finishes the proof of Step 5. 
\end{proof}

Proposition \ref{prop:raydeformation} follows immediately from Steps 4 and 5; 
item (\ref{item:prop:raydef:F}) follows thereby from the upper semicontinuity of fibre dimensions applied to  $Exc(F)\longrightarrow B$, where $Exc(F)$ denotes the exceptional locus of $F$.
\end{proof}

\begin{corollary} \label{cor:prop:raydef:F}
In the notation of Proposition \ref{prop:raydeformation}, for any $b\in U$ the natural map
$$
\overline{NE}(X_b)\longrightarrow \overline{NE}(\mathcal X_U/U)
$$
yields a one to one correspondence between the $(K_{X_b}+\Delta_b)$-trivial extremal faces of $X_b$ with the $(K_{\mathcal X_U}+\Delta_U)$-trivial extremal faces of $\overline{NE}(\mathcal X_U/U)$.
Moreover, if $F\colon\mathcal X_U\longrightarrow \mathcal Z$ denotes the contraction of a $(K_{\mathcal X_U}+\De_U)$-trivial
 extremal face, then the following holds:
 \begin{enumerate}
 \item for all $b\in U$, the restriction $F_b:=F|_{X_b}$ of $F$ is the contraction of the corresponding extremal face of $\overline{NE}(X_b)$; \label{item:cor:1}
 \item 
$\mathcal Z$ is $\Q$-factorial, if and only if one fibre $Z_b$ is $\Q$-factorial, if and only if all fibres $Z_b$ with $b\in U$ are $\Q$-factorial. \label{item:cor:2} 
 \end{enumerate} 
\end{corollary}

\begin{proof}
The first part of the corollary is immediate from Proposition \ref{prop:raydeformation}.
Item (\ref{item:cor:1}) follows by a similar argument as in Step 5 of the proof of Proposition \ref{prop:raydeformation}, and we leave the details to the reader.
Finally, item (\ref{item:cor:2}) follows from Lemma \ref{lem:Q-factorial} below and the fact that after shrinking $U$ we may assume that all exceptional divisors of $F$ are flat over $B$.
\end{proof}

\begin{lemma}\label{lem:Q-factorial}
Let $(X,\De)$ be a $\Q$-factorial klt pair and let $f \colon X \longrightarrow Y$ be a contraction of a $(K_X + \De)$-negative extremal face $\mathcal F \subset \overline{NE}(X)$. 
Let $\left\langle \mathcal F \right\rangle \subset N_1(X)$ be the linear subspace generated by $\mathcal F$ and assume it has dimension $r$. 
Then $Y$ is $\Q$-factorial if and only if there are exceptional divisors $E_1, \ldots, E_r$ such that the natural map
$$
\left\langle E_1, \ldots, E_r \right\rangle \longrightarrow \left\langle \mathcal F \right\rangle^\vee
$$
is an isomorphism.
\end{lemma}

\begin{proof} 
Assume that $Y$ is $\Q$-factorial. By the cone theorem we have that $f^*N^1(Y)=\left\langle \mathcal F \right\rangle^{\perp}$ has codimension $r$ in $N^1(X)$.  
Let $H_1, \ldots, H_r$ denote a basis of a complement.  
By the $\Q$-factoriality of $Y$, we can define $E_i:=f^*f_*H_i-H_i$. 
Note that each $E_i$ is exceptional and $E_1,\ldots ,E_r$ still generate a complement to $f^*N^1(Y)$. 
Hence,
$\left\langle E_1, \ldots, E_r \right\rangle \longrightarrow \left\langle \mathcal F \right\rangle^\vee$ is an isomorphism, as we want.

Conversely, suppose that there are exceptional divisors $E_1,\ldots, E_r$ such that the natural map $\left\langle E_1, \ldots, E_r \right\rangle \longrightarrow \left\langle \mathcal F \right\rangle^\vee$ is an isomorphism. 
If $D'$ is a Weil divisor on $Y$, denote by $D=f_*^{-1}D'$ its proper transform on $X$. 
Then by our assumptions there are rational numbers $a_1, \ldots,a_r$ such that $L:=D-\sum_i a_iE_i$ is a divisor which is trivial on $\left\langle \mathcal F \right\rangle$. 
By the cone theorem, $L=f^*D''$ for some $\Q$-Cartier divisor $D''$ on $Y$. Since each $E_i$ is exceptional, $D' = D''$ as we want. 
\end{proof}

\begin{remark} \label{rem:raydeformation}
In the case where $\Delta=0$, Koll\'ar and Mori proved that flops of threefolds deform in families.
That is, for any family of terminal threefolds $\mathcal X\longrightarrow S$ over the germ of a complex space $0\in S$, any flop $X_0\dashrightarrow X^+_0$ of the central fibre extends to a flop $\mathcal X\dashrightarrow \mathcal X^+$ of families, see \cite[Theorem 11.10]{KM92}. 
This result does not generalize to higher dimensions.
For instance, one can use the local Torelli theorem to construct families of hyperk\"ahler fourfolds where the fibre over a special point admits a flop which does not deform because the corresponding curve class does not deform to an algebraic cohomology class on nearby fibres.
Taking branched coverings, one can exhibit similar examples among families of varieties of general type. 
Therefore, the base change in Proposition \ref{prop:raydeformation} is necessary. 

The above mentioned examples show furthermore that there are smooth families of varieties of general type, such that the number of marked minimal models of the fibres is not a locally constant function on the base.  
In particular, the number of marked minimal models is not a topological invariant. 
\end{remark}

\section{Extending sequences of flops from one fibre to a family}

Consider a family $\pi:(\mathcal X,\Delta) \longrightarrow B$ as in Proposition \ref{prop:raydeformation}.
After shrinking $B$, we know by Proposition \ref{prop:raydeformation} and \cite{BCHM} that we can extend any flop of a fibre $(X_b,\Delta_b)$ to a flop of the total space.
We extend this to sequences of flops in Proposition \ref{prop:family} below; the proofs use some ideas from \cite[Section 5]{schreieder-tasin}. 

\begin{lemma} \label{lem:flopextension}
Let $\pi:(\mathcal X,\Delta) \longrightarrow B$ be a family of terminal pairs as in Proposition \ref{prop:raydeformation}.
If $g:(\mathcal X_\eta,\Delta_\eta)\dashrightarrow (\mathcal X_\eta^+,\Delta_\eta^+)$ is a flop of the generic fibre, then after shrinking $B$, $g$ extends to a flop of $(\mathcal X,\Delta) $ over $B$.
\end{lemma}
\begin{proof}
Let 
$$
\xymatrix{
\mathcal X_\eta \ar@{-->}[rr]^g  \ar[dr]_f & & \mathcal X_\eta^+\ar[dl]^{f^+} \\
& \mathcal Z_\eta 
}
$$
be the commutative diagram associated to the flop, and let $-H_\eta$ be an $f$-ample $\Q$-Cartier divisor on $\mathcal X_\eta$ such that $g_\ast H_\eta$ is an $f^+$-ample $\Q$-Cartier divisor.

After shrinking $B$, we may assume that the contraction $f:\mathcal X_\eta\longrightarrow \mathcal Z_\eta$, the relatively ample $\Q$-Cartier divisor $-H_\eta$ and the rational map $g$ extend over $B$.
We thus obtain a contraction $F:\mathcal X\longrightarrow \mathcal Z$, a $\Q$-Cartier divisor $H$ on $\mathcal X$ and a rational map $G:(\mathcal X,\Delta)\dashrightarrow (\mathcal X^+,\Delta^+)$.  
After shrinking $B$, we may by Proposition \ref{prop:raydeformation} assume that $F$ contracts a relative extremal ray.
It is then easy to see that $G$ defines a flop of $(\mathcal X,\Delta)$ over $B$.
By definition, $G$ restricts to $g$ on the generic fibre.
This concludes the lemma.
\end{proof}

\begin{proposition} \label{prop:family}
Let $\pi:(\mathcal X,\Delta) \longrightarrow B$ be a family of terminal pairs as in Proposition \ref{prop:raydeformation}.
Up to replacing $B$ by some sufficiently small Zariski open and dense subset, the following holds:
\begin{enumerate}
\item there are finitely many sequences of flops \label{item:prop:1} 
$
\phi^k:(\mathcal X,\Delta)\dashrightarrow (\mathcal X^k,\Delta^k)
$
over $B$, $k=1,\dots ,r$, such that each of the families $(\mathcal X^k,\Delta^k)\longrightarrow B$ satisfies the conclusion of Proposition \ref{prop:raydeformation} with $U=B$.
\item for any $b\in B$ and for any sequence of flops \label{item:prop:2} 
$
g:(X^k_b,\Delta^k_b)\dashrightarrow ((X^k_b)^+,(\Delta_b^k)^+),
$ 
there is some index $j$ such that $\phi^j\circ(\phi^k)^{-1}$ restricts to $g$ above $b\in B$. 
\end{enumerate} 
\end{proposition}

\begin{proof}
The generic fibre $(\mathcal X_\eta,\Delta_\eta)$ is a pair over the function field $\C(B)$ of $B$.
By \cite[Theorem E]{BCHM}, the base change of this pair to an algebraic closure $\overline{\C(B)}$ admits (up to isomorphism) only  
finitely many sequences of flops. 
This implies that also $(\mathcal X_\eta,\Delta_\eta)$ admits only finitely many sequences of flops. 
After shrinking $B$, we may thus by Lemma \ref{lem:flopextension} assume that there are finitely many sequences of flops
$$
\phi^k:(\mathcal X,\Delta)\dashrightarrow (\mathcal X^k,\Delta^k) 
$$
over $B$, $k=1,\dots ,r$, such that any sequence of flops of $(\mathcal X_\eta,\Delta_\eta)$ is given by restricting one of the $\phi^k$ to the generic fibre. 
Applying Proposition \ref{prop:raydeformation}, we may additionally assume that each $(\mathcal X^k,\Delta^k) $ satisfies the conclusion of Proposition \ref{prop:raydeformation} with $U=B$.
Shrinking further, we may also assume that whenever some composition $\phi^j\circ(\phi^k)^{-1}$ restricts to a flop of $(\mathcal X^k_\eta,\Delta^k_\eta)$, then $\phi^j\circ(\phi^k)^{-1}$ is a flop over $B$ (and not only a sequence of such).

In order to prove the proposition, it suffices to treat the case where $g$ is a single flop.
Using the existence of flops \cite[Corollary 1.4.1]{BCHM} and the fact that the conclusion of Proposition \ref{prop:raydeformation} holds for $U=B$, we see that $g$ extends to a flop
$$
G:(\mathcal X^k,\Delta^k)\dashrightarrow ((\mathcal X^k)^+,(\Delta^k)^+)
$$
over $B$.
Restricting to the generic fibre gives a flop $G_{\eta}$ and so we obtain a sequence of flops $G_\eta\circ \phi^k_\eta$ of  $(\mathcal X_\eta,\Delta_\eta)$.
By construction, there is some $j$, such that $\phi^j_\eta$ coincides with $G_\eta\circ \phi^k_\eta$.
Hence, $G$ and $\phi^j\circ(\phi^k)^{-1}$ coincide when restricted to the generic fibre.
In particular, $\phi^j\circ(\phi^k)^{-1}$ restricts to a flop on the generic fibre and so, by our assumptions, $\phi^j\circ(\phi^k)^{-1}$ is a flop over $B$.
Since $ (\mathcal X^k,\Delta^k)$ satisfies the conclusion of  Proposition \ref{prop:raydeformation} for $U=B$, any flop of that pair corresponds to a flat family of curves over $B$.
This implies $G=\phi^j\circ(\phi^k)^{-1}$, because both flops are determined by such a flat family of curves and we know that they coincide on the generic fibre. 
This concludes the proposition.
\end{proof}

\section{Terminalisations}

In the previous sections, we have only treated families of terminal pairs.
This is not very satisfactory, because the weak log canonical model of a terminal pair is in general not terminal but only klt.  
In this section we establish what is necessary to reduce the study of klt pairs to what we have proven for terminal pairs in Propositions \ref{prop:raydeformation} and \ref{prop:family} above. 
We start by recalling the following definition.

\begin{definition}\label{def:terminalisation}
Let $(X,\De)$ be a log pair.
A log pair $(X',\De')$ with a projective birational morphism $f \colon (X',\De') \longrightarrow (X,\De)$ is called a \emph{terminalisation} of $(X,\De)$ if $(X',\De')$ has terminal $\Q$-factorial singularities and $K_{X'}+\De'=f^*(K_X+\De)$. 
\end{definition}

Terminalisations of klt pairs always exist, see (the paragraph after) \cite[Corollary 1.4.3]{BCHM}, where they are called \emph{terminal models}. The following lemma is well-known.

\begin{lemma}\label{lem:discr}
Let $f \colon (X',\De') \longrightarrow (X,\De)$ be a terminalisation of a klt pair $(X,\De)$. 
Then the exceptional divisors of $f$ are exactly the exceptional divisors $E$ over $X$ such that $a(E,X,\De) \le 0$.
Moreover,
$$
\De'= \sum_{E \subset X' : a(E,X,\De) \le 0} -a(E,X,\De)E.
$$
\end{lemma}
%
 
\begin{lemma} \label{lem:flops+contr}
Let $(X,\Delta)$ be a klt pair with $K_X+\Delta$ big and nef, and consider its log canonical model $p:(X,\Delta)\longrightarrow (X^c,\Delta^c)$.
Let $q:(X^t,\Delta^t)\longrightarrow (X^c,\Delta^c)$ be a terminalisation.
Then there is a commutative diagram
$$
\xymatrix{
(X^t,\De^t)\ar@{-->}[r]^{\xi} \ar[d]^q & (X^+,\De^+) \ar[d]^\psi \\
 (X^c,\De^c)    &  (X,\De)  \ar[l]_p}
$$
where $\xi$ is a sequence of flops and $\psi$ is a contraction of a $(K_{X^+}+\Delta^+)$-trivial face. 
\end{lemma}
\begin{proof}
Let $\psi \colon (X^+,\Delta^+)\longrightarrow (X,\Delta)$ be a terminalization.
Since $K_{X^+}+\Delta^+=\psi^\ast(K_X+\Delta)$,  $\psi$ is the contraction of a $(K_{X^+}+\Delta^+)$-trivial face.

Note that $p\circ \psi\colon (X^+,\Delta^+)\longrightarrow (X^c,\Delta^c)$ is a terminalization of $(X^c,\Delta^c)$. 
It thus follows from Lemma \ref{lem:discr} that the natural birational map $\xi:X^t\dashrightarrow X^+$ satisfies $\xi_*\De^t=\De^+$.
Since $(X^t,\De^t)$ and $(X^+,\De^+)$ are terminal $\Q$-factorial and $K_{X^t}+\De^t$ and $K_{X^+}+\De^+$ are both nef, $\xi$ is a sequence of flops by  \cite[Theorem 1]{kawamata-flops}. 
This proves the lemma. 
\end{proof}

\begin{lemma}\label{lem:wlcm}
Let $(X,\De)$ be a klt pair of general type and $\phi \colon (X,\De) \dashrightarrow (X^c,\De^c)$ its log canonical model. 
A birational map $g \colon (X,\De) \dashrightarrow (Y,\Gamma)$ to a klt pair $(Y,\Gamma)$  is a weak log canonical model if and only if there is a terminalisation $q \colon (X^t,\De^t) \longrightarrow (X^c,\De^c)$ such that the induced map $\psi=g \circ \phi^{-1} \circ q \colon(X^t,\De^t) \longrightarrow (Y,\Gamma)$ is a birational morphism with the following properties:
\begin{itemize}
\item $\psi$ is the contraction of a $(K_{X^t}+\De^t)$-trivial extremal face of $\overline{NE}(X^t)$;
\item if $E$ is a divisor on $X^t$ which is contracted by $\phi^{-1}\circ q\colon X^t \dashrightarrow X$ or for which $a(E,X,\De) \ne a(E,X^c,\De^c)$, then $E$ is contracted by $\psi$.
\end{itemize} 
\end{lemma}

\begin{proof} 
We have the following diagram:
$$
\xymatrix{
(X,\De) \ar@{-->}[d]^\phi  & (X^t,\De^t) \ar[dl]_q \ar[d]^\psi \\
 (X^c,\De^c)    &  (Y,\Gamma).  \ar[l]_p}
$$

Assume first that $g \colon (X,\De) \dashrightarrow (Y,\Gamma)$ is a weak log canonical model.
By the uniqueness of log canonical models, we know that $p:=\phi\circ g^{-1} \colon (Y,\Gamma) \longrightarrow (X^c,\De^c)$ is the log canonical model of $(Y,\Gamma)$ and a terminalisation $\psi \colon (X^t,\De^t) \longrightarrow (Y,\Gamma)$ induces a terminalisation $q:=p \circ \psi \colon (X^t,\De^t) \longrightarrow (X^c,\De^c)$ of $(X^c,\De^c)$.
The morphism $\psi$ is thus the contraction of a $(K_{X^t}+\De^t)$-trivial extremal face of $\overline{NE}(X^t)$ and if a divisor $E$ on $X^t$ is contracted by $\phi^{-1} \circ q$, then it is also contracted by $\psi=g \circ \phi^{-1} \circ q$, because $g$ is a contraction.
Moreover, if $E$ is a divisor on $X^t$ which is not contracted by $\psi$, then its pushforward on $X$ is not contracted by $g$ and so $a(E,X,\De) = a(E,Y,\Gamma) = a(E,X^c,\De^c)$.


Viceversa, assume  that we have a birational map $g \colon (X,\De) \dashrightarrow (Y,\Gamma)$ to a klt pair $(Y,\Gamma)$ and a terminalisation $q \colon (X^t,\De^t) \longrightarrow (X^c,\De^c)$ as in the statement of the lemma. 
Since $\psi$ is the contraction of a $(K_{X^t}+\De^t)$-trivial extremal face, we have that $K_Y+\Gamma$ is big and nef and $p:=\phi\circ g^{-1} \colon (Y,\Gamma) \longrightarrow (X^c,\De^c)$ is the log canonical model of $(Y,\Gamma)$.  We need to prove that $g$ is a $(K_X+\De)$-non-positive contraction such that $g_*\De=\Gamma$.

The fact that $g$ is a contraction follows from the fact that any divisor $E$ on $X^t$ which is contracted by $\phi^{-1}\circ q\colon X^t \dashrightarrow X$ is contracted also by $\psi$.  
Let $E$ be a prime divisor on $Y$. Then $a(E,X^c,\De^c)=a(E,Y,\Gamma)=a(E,X,\De)$ and so $\textrm{mult}_\Gamma(E)=\textrm{mult}_\Delta(E')$ where $E'$ is the strict transform of $E$ on $X$ via $g$ and hence $g_*\De=\Gamma$. 
Finally, the contraction $g$ is $(K_X+\De)$-non-positive because any divisor $E$ contracted by $g$ is also contracted by $\phi$ and so $a(E,X,\De) \le a(E,X^c,\De^c) = a(E,Y,\Gamma)$.
\end{proof}
 
\begin{lemma}\label{lem:X=klt}  
Let $B$ be a quasi-projective variety, and let $\pi \colon (\mathcal X,\Delta)\longrightarrow B$ be a family of klt (log canonical) pairs. 
Then, for any sufficiently small dense open subset $U\subset B$, the base change $\mathcal X_U:=\mathcal X\times_BU$ with boundary $\Delta_U:=\Delta\times_X U$ form a klt (resp.\ log canonical) pair $(\mathcal X_U,\Delta_U)$.  
\end{lemma}
\begin{proof}
Since the fibres of $\pi$ are normal, the normalization of $\mathcal X$ is an isomorphism over a general point of $B$. 
After shrinking $B$, we may thus assume that $\mathcal X$ is normal.

We now show that, up to shrinking $B$, $K_{\mathcal X}+ \De$ is a $\Q$-Cartier divisor.
The geometric generic fibre $(X_{\overline \eta},\De_{\overline \eta})$ is a klt (resp. lc) pair, so there exists a positive integer $a$ such that $a(K_{X_{\overline \eta}}+\De_{\overline \eta})$ is Cartier.
This Cartier divisor is defined over some finite extension of the function field $\textbf{k}(B)$.
Therefore, there is a dominant morphism $p : B'\longrightarrow B$, finite onto its image, such that the base change $\pi' : (\mathcal X',\De')\longrightarrow B'$ carries a Cartier divisor $L$ which restricts to $a(K_{X'_{\overline \eta}}+\De'_{\overline \eta})$ on the geometric generic fibre. 
After shrinking $B$ and $B'$, we may further assume that $B$ and $B'$ are smooth, $p : B'\longrightarrow B$ is proper and \'etale, and $L=a(K_{\mathcal X'}+\De')$, i.e.\ $K_{\mathcal X'}+\De'$ is $\Q$-Cartier.
Considering the pushforward of  $K_{\mathcal X'}+\De'$ to $\mathcal X$ shows that $K_{\mathcal X}+\De$ is $\Q$-Cartier, see for instance the proof of \cite[Lemma 5.16]{kollar-mori}.

The lemma follows now from \cite[Theorem 7.5]{kollar-pairs}.
\end{proof}	

Terminalisations exist in families in the following sense, see also \cite[Proposition 2.5]{HX14}.

\begin{lemma}\label{lem:terminfamilies} 
Let $p \colon (\mathcal Y, \Gamma) \longrightarrow S$ be a projective family of klt pairs.
Then there is a surjection $q \colon B\twoheadrightarrow S$, which on each component of $B$ is finite onto its image, and a projective family of pairs $\pi \colon (\mathcal X,\De) \longrightarrow B$, such that for all $b\in B$, $(X_b,\De_b) \longrightarrow (Y_{q(b)},\Gamma_{q(b)})$ is a terminalisation of $(Y_{q(b)},\Gamma_{q(b)})$. 
\end{lemma}
\begin{proof}
It suffices to treat the special case where $S$ is irreducible and we prove the lemma by induction on $N:=\dim(S)$.
If $N=0$, the statement follows from the existence of a terminalisation for a klt pair, see \cite[Corollary 1.4.3]{BCHM}.

Let now $N>0$. 
By induction, we may without loss of generality replace $S$ by any Zariski open and dense subset; the complement has lower dimension and so we know the result there by induction.

By Lemma \ref{lem:X=klt} we can assume that $(\mathcal Y, \Gamma)$ is a klt pair. 
Let $f:(\mathcal Y',\Gamma') \longrightarrow (\mathcal Y, \Gamma)$ be a log resolution, where we wrote
$$
K_{\mathcal Y'} + \Gamma' = f^*(K_{\mathcal Y} + \Gamma) + E
$$
such that $\Gamma'$ and $E$ are effective divisors with no common component. We have that  $(\mathcal Y' ,\Gamma')$ is klt.

Up to shrinking $S$, we may assume that $\mathcal Y'$ and $S$ are smooth and quasi-projective and the morphism $p \circ f$ is smooth. We can also assume that $f$ is fibrewise a log resolution such that
$$
K_{Y'_s} + \Gamma'_s = f^*(K_{ Y_s} + \Gamma_s) + E_s
$$
for any $s \in S$, where $\Gamma'_s$ and $E_s$ are effective with no common component.

After shrinking $S$, there is by Lemma \ref{lem:monodromy} a variety $B$ and a finite morphism $B\longrightarrow S$, such that the base change $ \mathcal W':=\mathcal Y'\times_SB\longrightarrow B$ has the property that for a very general $t\in B$, the restriction map
$
N^1(\mathcal W')\longrightarrow N^1(W'_t)
$
is onto.
By Lemma \ref{lem:X=klt} we may assume that $(\mathcal W', \Sigma')$ is klt, where $\Sigma'=\Gamma' \times_S B$.
The same is true for $(\mathcal W, \Sigma)$ where  $\mathcal W := \mathcal Y\times_SB\longrightarrow B$ and $\Sigma:= \Gamma \times_S B$.

Shrinking further, we may additionally assume that $\mathcal W'$ and $B$ are smooth and quasi-projective, that $(\mathcal W', \Sigma')$ and $(\mathcal W, \Sigma)$ are projective families of klt pairs over $B$ and that $(\mathcal W', \Sigma')$ is fibrewise a log-resolution of $(\mathcal W, \Sigma)$.
We have the following diagram, where the horizontal rows are finite morphisms and the vertical maps have connected fibres.
$$
\xymatrix{
(\mathcal W', \Sigma') \ar[r] \ar[d] & (\mathcal Y',\Gamma') \ar[d]^f \\
 (\mathcal W,\Sigma)  \ar[r] \ar[d] & (\mathcal Y,\Gamma)   \ar[d]^p      \\
 B \ar[r] &  S
}
$$

By \cite{BCHM}, we can run a relative log MMP of $(\mathcal W', \Sigma')$ over $\mathcal W$, and arrive at a relatively log terminal model 
$$
G \colon (\mathcal W', \Sigma') \dashrightarrow \mathcal (X,\De)
$$ 
over $\mathcal W$, $g\colon (\mathcal X, \De)\longrightarrow \mathcal W$.

The pair $(\mathcal X,\De)$ is a terminalisation of $(\mathcal W, \Sigma)$ (see the proof of \cite[Corollary 1.4.3]{BCHM}). 
In particular, it is terminal and $\Q$-factorial and so applying Bertini's theorem after shrinking $B$ we may assume that all fibres of $\pi \colon (\mathcal X,\De)\longrightarrow B$ are terminal and that
$$
K_{X_b}+\De_b=g^*(K_{W_b}+\Sigma_b)
$$
for any $b \in B$.

In order to conclude the proof of the lemma, we just need to show that over an open and dense Zariski subset of $B$, the fibres of $\pi$ are $\Q$-factorial. 
By the openess of $\Q$-factoriality in families of terminal varieties (see \cite[Theorem 12.1.10]{KM92}), we are done if we can find one fibre of $\pi$ that is $\Q$-factorial.
For this it suffices to show that the restriction of $G$ to a very general fibre $(W'_t,\Sigma'_t)$ is given by (some steps of) a log MMP for that fibre.
The latter is an easy consequence of the fact that $N^1(\mathcal W')\longrightarrow N^1(W'_t)$ is onto and $t\in B$ is general, see for example the proof of \cite[Theorem 4.1]{defernex-hacon}.
 \end{proof}

\begin{remark} 
In order to ensure in Lemma \ref{lem:terminfamilies} that the fibres of a smooth family $\mathcal X\longrightarrow B$ stay $\Q$-factorial after running a relative MMP,  
it is necessary to perform a base change which trivializes the monodromy on $N_1(X_t)$ for very general $t\in B$ (or equivalently, such that $N^1(\mathcal X)\longrightarrow N^1(X_t)$ is onto, see Section \ref{subsubsec:monodromy}), cf.\ \cite[Remark 12.4.3]{KM92}. 
\end{remark}

\section{Proof of the main results} 

\begin{proof}[Proof of Theorem \ref{thm:stratification}]
Let $\pi \colon \mathcal X\longrightarrow B $ be a family of complex projective varieties such that the resolution of each fibre is of general type. 
By induction on the dimension of $B$, it suffices to prove that the number of marked minimal models of the fibres of $\pi$ are locally constant on some Zariski open and dense subset.
Moreover, it suffices to treat the case where $B$ is irreducible and we may replace $B$ by any dominant base change $B'\twoheadrightarrow B$. 
Replacing $\mathcal X$ by the relative canoncial model of (some resolution of) $\mathcal X$ over $B$ and shrinking $B$ if necessary, we may assume that $\pi$ is a family of canonical models.
The terminalisation of a canonical model is a minimal model.
Applying Lemma \ref{lem:terminfamilies}, we may therefore assume that $\pi$ is a family of minimal models of general type. 
By Proposition \ref{prop:family}, we may also assume that there are finitely many projective families of minimal models $\pi^{k} \colon \mathcal X^k\longrightarrow B$, with $\mathcal X^1=\mathcal X$, and such that there are sequences of flops of families
$\phi^k \colon \mathcal X\dashrightarrow \mathcal X^k$ such that item (\ref{item:prop:2}) of Proposition \ref{prop:family} holds.
That is, for any sequence of flops $g \colon X\dashrightarrow X^+$ of any fibre of $\pi^k$, there is some $j$ such that the restriction of the composition $\phi^j\circ (\phi^k)^{-1}$ to $X\subseteq \mathcal X^k$ coincides with $g$.
Since marked minimal models are connected by sequences of flops (see Theorem \ref{thm:minmod-connected}), Theorem \ref{thm:stratification} follows therefore from the easy fact that the locus of points $b\in B$, such that the restriction of $\phi^j\circ (\phi^k)^{-1}$ to the fibre $X^k_b$ is an isomorphism, is constructible in the Zariski topology of $B$.
\end{proof}

\begin{proof}[Proof of Corollary \ref{cor:nrofmodels}]
By \cite{HM06,takayama,tsuji}, smooth complex projective varieties of general type, of given dimension and bounded volume form a birationally bounded family.   
Corollary \ref{cor:nrofmodels} follows then by Theorem \ref{thm:stratification} and the fact that any constructible function on a scheme of finite type is bounded. 
\end{proof}

\begin{proof}[Proof of Theorem \ref{thm:stratification:pairs}] 
By induction on the dimension of $B$, it suffices to prove that the number of weak log canonical models of the fibres of $\pi$ is locally constant on some Zariski open and dense subset of $B$.
To prove this, we may assume that $B$ is irreducible and we are allowed to perform arbitrary dominant base changes $B'\longrightarrow B$.

By Lemma \ref{lem:X=klt} we may assume that $(\mathcal X,\De)$ is a klt pair.
Let $\phi \colon (\mathcal X, \De) \dashrightarrow (\mathcal X^c, \De^c)$  be its relative log canonical over $B$ and denote with $\pi^c \colon (\mathcal X^c, \De^c) \longrightarrow B$ the morphism to $B$.
Up to shrinking $B$ we may assume that $(X_b,\De_b) \dashrightarrow (X_b^c,\De_b^c)$ is a log canonical model for any $b \in B$. 

Let $\pi^t \colon  (\mathcal X^t, \De^t) \longrightarrow B$ be a family obtained applying Lemma \ref{lem:terminfamilies} to $\pi^c$. 
We apply Proposition \ref{prop:family} to the family $\pi^t \colon  (\mathcal X^t, \De^t) \longrightarrow B$ to get projective families $\pi^k \colon (\mathcal X^k,\Delta^k)\longrightarrow B$ of terminal and $\Q$-factorial pairs $(X^k_b,\Delta^k_b)$ with $K_{X^k_b}+\Delta^k_b$ big and nef, where $k=1,\dots ,r$, such that item (\ref{item:prop:1}) and (\ref{item:prop:2}) of Proposition \ref{prop:family}, as well as Proposition \ref{prop:raydeformation} with $U=B$ hold.
Let 
$$
\psi^{k,j}\colon (\mathcal X^{k},\Delta^{k})\longrightarrow (\mathcal X^{k,j},\Delta^{k,j})
$$
denote the contractions of the finitely many $(K_{\mathcal X^k}+\Delta^k)$-trivial extremal faces of $\overline{NE}(\mathcal X^k/B)$, where $j=1,\dots ,s(k)$; for simplicity, we include the trivial face in this list and so we may assume  $(\mathcal X^{k,1},\Delta^{k,1})= (\mathcal X^{k},\Delta^{k})$ and $\psi^{k,1}=\id$. Denote by $\pi^{k,j}\colon (\mathcal X^{k,j},\Delta^{k,j})\longrightarrow B$ the natural morphisms to $B$.

Up to shrinking $B$, we may assume that any component of the loci contracted by $\phi$ and $\psi^{k,j}$ is dominant and flat over $B$.
In particular, we can assume that for any point $b \in B$, if $E_b$ is a prime divisor contracted by the induced map $(X^{k})_b \dashrightarrow X_b$, then there is a prime divisor $E \subset \mathcal X^k$ which is contracted by $\mathcal X^{k} \dashrightarrow \mathcal X$ and restricts to $E_b$ on $X_b$.
Finally, we can assume that for any  exceptional divisor $E$ of $\psi^{k,j}$, we have $a(E,\mathcal X,\De)=a(E_b,X_b,\De_b)$ and $a(E,\mathcal X^c,\De^c)=a(E_b,X^c_b,\De^c_b)$ for any $b\in B$.

Let $(X_b,\De_b)$ be a fibre of $\pi$ over a point $b \in B$. By Lemma \ref{lem:flops+contr} and Corollary \ref{cor:prop:raydef:F}, any weak log canonical model of $(X_b,\De_b)$ is a fibre over $b$ of some of the families $\pi^{k,j}$. 
More precisely, by our assumptions and by Lemma \ref{lem:wlcm}, a fibre over $b$ of a family $\pi^{k,j}$ is a weak log canonical model of $(X_b,\De_b)$ if and only if $\psi^{k,j}$ contracts any divisor $E$ on $\mathcal X^{k}$ such that $E$ is contracted by the induced map $\mathcal X^{k} \dashrightarrow \mathcal X$ or such that $a(E,\mathcal X, \De) \ne a(E,\mathcal X^c, \De^c)$.
Let us consider only the families that satisfy these conditions and denote by $g^{k,j} \colon (\mathcal X,\De) \dashrightarrow (\mathcal X^{k,j}, \De^{k,j})$ the induced maps, which are exactly the weak log canonical models of the family $(\mathcal X,\De)$ over $B$. 
Theorem \ref{thm:stratification:pairs} follows now from the fact that the locus of points $b\in B$, such that the restriction of $g^{k,j}\circ (g^{h,l})^{-1}$ to the fibre $X^{h,l}_b$ is an isomorphism, is constructible in the Zariski topology of $B$. 
\end{proof}

\begin{remark} \label{rem:stratification:Q-factorial}
Consider $\pi^{k,j}\colon \mathcal X^{k,j} \longrightarrow B$ from the proof of Theorem \ref{thm:stratification:pairs}.
Up to possibly shrinking $B$, Corollary \ref{cor:prop:raydef:F} implies that one fibre of $\pi^{k,j}$ is $\Q$-factorial if and only if all fibres have that property.
This shows that Theorem \ref{thm:stratification:pairs} remains true if we count only $\Q$-factorial weak log canoncial models.
\end{remark}

\begin{proof}[Proof of Theorem \ref{thm:boundednesspairs}] 
By Theorem \ref{thm:boundeed:HMX} the set $\mathfrak{F}_{klt}$ of all $(Y,D)$ such that $Y$ is a  projective variety of dimension $n$, $(Y,D)$ is a klt pair, the coefficients of $D$ belong to $I$, $K_Y+ D$ is ample and $(K_Y+D)^n=c$, is bounded.
Let $p \colon (\mathcal Y, \Gamma) \longrightarrow S$ be a projective family of klt pairs (over a scheme $S$ of finite type over $\C$) such that any element of $\mathfrak F_{klt}$ is isomorphic to a fibre of $p$.

By Noetherian induction and Lemma \ref{lem:terminfamilies}, we get a projective family of terminal pairs $\pi\colon  (\mathcal X, \De) \longrightarrow B$ over a base $B$ of finite type with the following property: 
for any projective klt pair $(Y,D) \in \mathfrak F_{klt}$  there is some $b\in B$ such that $(X_b,\De_b)$ is a terminalisation of $(Y,D)$. 
Let $B_1,\dots ,B_s$ denote the components of $B$, and let $(\mathcal X_i, \De_i)$ denote the component of $(\mathcal X,\De)$ over $B_i$.
Applying Proposition \ref{prop:family} to the family of terminal pairs $(\mathcal X_i, \De_i)$, we may additionally assume that there are projective families $\pi_i^k \colon (\mathcal X_i^k,\Delta_i^k)\longrightarrow B_i$ of terminal and $\Q$-factorial pairs $({(X_i^k)}_b,{(\Delta_i^k)}_b)$ with $K_{{(X_i^k)}_b}+{(\Delta_i^k)}_b$ big and nef, obtained from sequences of flops $(\mathcal X_i, \De_i) \dashrightarrow (\mathcal X_i^k,\Delta_i^k)$, where $k=1,\dots ,r(i)$, such that item (\ref{item:prop:1}) and (\ref{item:prop:2}) of Proposition \ref{prop:family}, as well as  Proposition \ref{prop:raydeformation} with $U=B_i$ hold.
Let 
$$
\psi_i^{k,j}\colon (\mathcal X^{k}_i,\Delta_i^{k})\longrightarrow (\mathcal X^{k,j}_i,\Delta_i^{k,j})
$$
denote the contractions of the finitely many $(K_{\mathcal X_i^k}+\Delta_i^k)$-trivial extremal faces of $\overline{NE}(\mathcal X_i^k/B_i)$, where $j=1,\dots ,s(i,k)$; for simplicity, we include the trivial face in this list and so we may assume  $(\mathcal X^{k,1}_i,\Delta_i^{k,1})= (\mathcal X^{k}_i,\Delta_i^{k})$ and $\psi_i^{k,1}=\id$. 
By item (\ref{item:cor:1}) in Corollary \ref{cor:prop:raydef:F} and Lemma \ref{lem:flops+contr}, the disjoint union of all $\pi_i^{k,j}\colon (\mathcal X^{k,j}_i,\Delta_i^{k,j}) \longrightarrow B_i$ parametrizes all pairs $(X,\Delta)\in \mathfrak F$. 
To conclude, we just note that up to shrinking the $B_i$'s, we can assume by item (\ref{item:cor:2}) in Corollary \ref{cor:prop:raydef:F} that one fibre of $\pi_i^{k,j}\colon \mathcal X^{k,j}_i \longrightarrow B_i$ is $\Q$-factorial, if and only if each fibre is $\Q$-factorial, if and only if $\mathcal X^{k,j}_i$ is $\Q$-factorial.
This proves Theorem \ref{thm:boundednesspairs}. 
\end{proof}

\begin{proof}[Proof of Corollary \ref{cor:nrofmodels:pairs}]	 
Let $n=\dim X$, $c=\operatorname{vol}(X,\Delta)$ and $I$ be the set of coefficients of $\Delta$.  
As we have seen in the proof of Theorem \ref{thm:boundednesspairs}, there is a projective family of terminal pairs $\pi^t\colon  (\mathcal X^t, \De^t) \longrightarrow B$ over a base $B$ of finite type with the following property: 
for any projective klt pair $(Y,D)$ of dimension $n$ with $K_Y+D$ ample, $(K_Y+D)^n=c$ and such that the coefficients of $D$ belong to $I$, there is some $b\in B$ such that $(X^t_b,\De^t_b)$ is a terminalisation of $(Y,D)$. 
Since any constructible function on $B$ is bounded, applying Theorem \ref{thm:stratification:pairs} to this family shows that the number of weak log canonical models of the fibres of $\pi^t$ is bounded by a constant that depends only on $n$, $I$ and $c$.

In order to prove the corollary, it thus suffices to see that the number of weak log canonical models of $(X,\Delta)$ is bounded from above by the number of weak log canonical models of any terminalisation $q:(X^t,\Delta^t) \to (X^c,\Delta^c)$ of the canonical model $\alpha:(X,\De)\dashrightarrow (X^c,\Delta^c)$ of $(X,\Delta)$.
To prove the latter, let $f\colon (X,\Delta)\dashrightarrow (Y,\Gamma)$ be a weak log canonical model of $(X,\Delta)$.
Since $K_Y+\Gamma$ is big and nef, Lemma \ref{lem:flops+contr} implies that $(Y,\Gamma)$ is obtained via a sequence of flops $\xi\colon (X^t,\Delta^t) \dashrightarrow (X^+,\Delta^+)$ followed by a contraction of a $(K_{X^+}+\Delta^+)$-trivial face $\psi\colon (X^+,\Delta^+) \rightarrow (Y,\Gamma)$.    
In particular, $\psi\circ \xi\colon (X^t,\Delta^t) \dashrightarrow (Y,\Gamma)$ is a weak log canonical model of $(X^t,\Delta^t)$.
Moreover, since $f= \psi\circ \xi \circ q^{-1}\circ \alpha$, two weak log canonical models of $(X,\Delta)$ are isomorphic, if and only if the induced weak log canonical models of $(X^t,\De^t)$ are isomorphic.  
Therefore, the number of weak log canonical models of $(X,\Delta)$ is bounded from above by the number of weak log canonical models of $(X^t,\Delta^t)$, as we want.
This concludes the corollary. 
\end{proof}

\begin{proof}[Proof of Corollary \ref{cor:minmodbdd}]
By \cite{HM06,takayama,tsuji}, the set
$$
\{ \operatorname{vol}(X) : X \mbox{ is a smooth projective variety of dimension $n$}\}
$$	
is discrete.
The corollary follows then from Theorem \ref{thm:boundednesspairs} and the fact that terminality is an open condition in flat families. 
\end{proof}

\begin{proof}[Proof of Corollary \ref{cor:finitetopspace}]
To prove that the set of pairs $(X,D)\in \mathfrak F$ is finite up to homeomorphisms consider the family $\pi \colon (\mathcal X,D)\longrightarrow B$ given by Theorem \ref{thm:boundednesspairs}.
By \cite[Th\'eor\'eme 2.2]{verdier}, there is  a Whitney stratification of $\mathcal X$ such that $D$ is a union of strata.
It thus follows from \cite[Th\'eor\'eme 3.3 and 4.14]{verdier} that up to replacing $B$ by the strata of some stratification, we may assume that locally on $B$, $\pi$ is topologically trivial.  
In particular, for any two points $b,b'\in B$ in the same connected component of $B$, there is a homeomorphism between $X_b^{an}$ and $X_{b'}^{an}$ which maps $D_b^{an}$ to $D_{b'}^{an}$, as we want.

The same proof works for the case of minimal models of given dimension and bounded volume considering the family given by Corollary \ref{cor:minmodbdd}.
\end{proof}

\begin{proof}[Proof of Corollary \ref{cor:finitedeftype}]
Since smoothness in families is a Zariski open condition in the base, the corollary follows immediately from Corollary \ref{cor:minmodbdd}.
\end{proof}

\begin{proof}[Proof of Corollary \ref{cor:CY}]
Let $X$ be a projective variety with klt singularities, $K_X\equiv 0$ and let $L$ be a big and nef line bundle.
By \cite[Theorem 1]{kollar-effectivebpf}, $2(n+2)!(n+1)L$ is basepoint free.
By Bertini's theorem, there is a smooth divisor $D\in |2(n+2)!(n+1)L|$, and we put $\Delta:=\frac{1}{2(n+2)!(n+1)}D$.
Then, $(X,\Delta)$ is a klt pair with $K_X+\Delta$ big and nef, and $(K_X+\Delta)^{n}=L^n$.
Since $L^n$ is an integer and the coefficients of $\Delta$ depend only on the dimension of $X$, the corollary follows from Theorem \ref{thm:boundednesspairs}.
\end{proof}

\section{Open Questions}

In view of Theorem \ref{thm:stratification}, it is natural to pose the following question.

\begin{question}
Is the number of marked minimal models an upper semi-continuous function on the base of families of complex projective varieties of general type? 
\end{question}

Recall from Remark \ref{rem:raydeformation} that the number of marked minimal models is not a topological invariant.
In view of Corollary \ref{cor:betti}, it is nonetheless natural to pose the following higher dimensional version of the conjecture of Cascini and Lazi\'c \cite{cascini-lazic}.

\begin{question} \label{question:CL}
Is the number of marked minimal models of a smooth complex projective variety of general type bounded in terms of the underlying topological space?
\end{question}

By Theorem \ref{thm:boundednesspairs}, Question \ref{question:CL} is related to the following question, which by \cite{cascini-tasin} is known to have an affirmative answer in dimension at most three.

\begin{question}
Is the volume of a smooth complex projective variety of general type bounded in terms of the underlying topological space?
\end{question}

\section*{Acknowledgement}  
We are grateful to P.\ Cascini for useful conversations and, in particular, for encouraging us to generalize our results from the case of threefolds to arbitrary dimensions.
We thank V.\ Lazi\'c for many useful comments on an earlier version of this paper, R.\ Svaldi for helpful conversations and V.\ Tosatti for a question which motivated Corollary \ref{cor:CY}.
We also thank the referee for spotting an inaccuracy in a previous version of Theorem \ref{thm:boundednesspairs}.

Parts of the results of this article were conceived when the third author was supported by the DFG Emmy Noether-Nachwuchsgruppe ``Gute Strukturen in der h\"oherdimensionalen birationalen Geometrie''.
The second author is member of the SFB/TR 45 and thanks the Universit\`a Roma Tre for hospitality, where parts of this project were carried out.

\end{document}